\documentclass[fleqn,12pt,twoside]{article}

\usepackage[headings]{espcrc1}
\readRCS
$Id: espcrc1.tex,v 1.2 2004/02/24 11:22:11 spepping Exp $
\usepackage{graphicx}
\usepackage[figuresright]{rotating}

\newtheorem{theorem}{Theorem}

\newtheorem{lemma}[theorem]{Lemma}

\newtheorem{problem}{Problem}

\newenvironment{proof}[1][Proof.]{\begin{trivlist}
\item[\hskip \labelsep {\bfseries #1}]}{\end{trivlist}}

\newcommand{\AmS}{{\protect\the\textfont2
  A\kern-.1667em\lower.5ex\hbox{M}\kern-.125emS}}

\usepackage{amsmath}
\usepackage{amsfonts}
\usepackage{amssymb}
\numberwithin{theorem}{section}

\title{Complexity results on locally-balanced $2$-partitions of graphs}

\author{Aram H. Gharibyan\address[MCSD]{Department of Informatics and Applied Mathematics,\\
Yerevan State University, 0025, Armenia\\
}\thanks{email: aramgharibyan@gmail.com},
    Petros A. Petrosyan\addressmark[MCSD]\thanks{email: petros\_petrosyan@ysu.am}}
    
\runtitle{Complexity results on locally-balanced 2-partitions of graphs
}\runauthor{Aram H. Gharibyan, Petros A.
Petrosyan}

\begin{document}

\maketitle

\begin{abstract}
A \emph{$2$-partition of a graph $G$} is a function $f:V(G)\rightarrow \{0,1\}$. A $2$-partition $f$ of a graph $G$ is a \emph{locally-balanced with an open neighborhood} if for every $v\in V(G)$, 
$$\left\vert \vert \{u\in N_{G}(v)\colon\,f(u)=0\}\vert - \vert
\{u\in N_{G}(v)\colon\,f(u)=1\}\vert \right\vert\leq 1.$$

A $2$-partition $f^{\prime}$ of a graph $G$ is a \emph{locally-balanced with a closed neighborhood} if for every $v\in V(G)$, 
$$\left\vert \vert \{u\in N_{G}[v]\colon\,f^{\prime}(u)=0\}\vert - \vert \{u\in N_{G}[v]\colon\,f^{\prime}(u)=1\}\vert \right\vert\leq 1.$$

In this paper we prove that the problem of the existence of locally-balanced $2$-partition with an open (closed) neighborhood is $NP$-complete for some restricted classes of graphs. In particular, we show that the problem of deciding if a given graph has a locally-balanced $2$-partition with an open neighborhood is $NP$-complete for biregular bipartite graphs and even bipartite graphs with maximum degree $4$, and the problem of deciding if a given graph has a locally-balanced $2$-partition with a closed neighborhood is $NP$-complete even for subcubic bipartite graphs and odd graphs with maximum degree $3$. Last results prove a conjecture of Balikyan and Kamalian.\\ 

\end{abstract}


\section{Introduction}
\label{intro}
In this paper all graphs are finite, undirected, and have no loops or multiple edges, unless otherwise stated. Let $V(G)$ and $E(G)$ denote the sets of vertices and edges of a graph $G$, respectively. The set of neighbors of a vertex $v$ in $G$ is denoted by $N_{G}(v)$. Let $N_{G}[v]= N_{G}(v)\cup \{v\}$. The degree of a vertex $v\in V(G)$ is denoted by $d_{G}(v)$ and the maximum degree of vertices in $G$ by $\Delta(G)$. A graph $G$ is \emph{even} (\emph{odd}) if the degree of every vertex of $G$ is even (odd). A bipartite graph is \emph{$(a,b)$-biregular} if all vertices in one part have degree $a$ and all vertices in the other part have degree $b$. The terms and concepts that we do not define can be found in \cite{ChartZhang,West}.

One of the important areas of research in graph theory is the one of graph partition problems. The main motivation for that is the fact of the existence of many different applications of graph partition problems such as VLSI design, parallel computing, task scheduling, clustering and detection of communities, cliques and cores in complex networks, etc., and the existence of many problems in graph theory which can be formulated as graph partition problems (factorization problems, coloring problems, clustering problems, problems of Ramsey Theory). For example, the problem of finding the arboricity of a graph is the one of decomposing of a graph into minimum number of forests, the problem of determining of the chromatic number of a graph is the problem of decomposing of a graph into minimum number of independent sets, and the problem of determining of the chromatic index of a graph is the one of decomposing of a graph into minimum number of matchings. Some other applications of graph partition problems can be found in \cite{AndreevRacke}.

The concept of locally-balanced $2$-partition of graphs was introduced by Balikyan and Kamalian \cite{BalKamNP1} in 2005 and motivated by the problems concerning a distribution of influences of two different powers, which minimizes the probability of conflicts. The subjects of a modelling system may or may not have an ability of self-defense. Locally-balanced $2$-partitions of graphs also can be considered as a special case of equitable colorings of hypergraphs \cite{Berge}. In \cite{Berge}, Berge  obtained some sufficient conditions for the existence of equitable colorings of hypergraphs. Ghouila-Houri \cite{Ghouila-Houri} characterized unimodular hypergraphs in terms of partial equitable colorings and proved that a hypergraph $H=(V,E)$ is unimodular if and only if for each
$V_{0}\subseteq V$ there is a $2$-coloring $\alpha:V_{0}\rightarrow \{0,1\}$ such that for every $e\in E$, $\left\vert|e\cap \alpha^{-1}(0)|-|e\cap \alpha^{-1}(1)|\right\vert\leq 1$. In \cite{HajSzem,Kostochka,Meyer,deWerra}, it was considered the problems of the existence and construction of proper vertex-coloring of a graph for which the number of vertices in any two color classes differ by at most one. In \cite{Kratochvil}, $2$-vertex-colorings of graphs were considered for which each vertex is adjacent to the same number of vertices of every color. In particular, Kratochvil \cite{Kratochvil} proved that the problem of the existence of such a coloring is $NP$-complete even for the $(2p,2q)$-biregular ($p,q\geq 2$) bipartite graphs. Moreover, he also showed that the problem of the existence of the aforementioned coloring for the $(2,2q)$-biregular ($q\geq 2$) bipartite graphs can be solved in polynomial time. Gerber and Kobler \cite{GerKob1,GerKob2} suggested to consider the problem of deciding if a given graph has a $2$-partition with nonempty parts such that each vertex has at least as many neighbors in its part as in the other part. In \cite{BazTuzaVand}, it was proved that the problem is $NP$-complete. 
In \cite{BalKamNP1}, Balikyan and Kamalian proved that the problem of existence of locally-balanced $2$-partition with an open neighborhood of bipartite graphs with maximum degree $3$ is $NP$-complete. In 2006, the similar result for locally-balanced $2$-partitions with a closed neighborhood was also proved \cite{BalKamNP2}. In \cite{Bal,BalKam}, the necessary and sufficient conditions for the existence of locally-balanced $2$-partitions of trees were obtained. In \cite{GharibPet}, Gharibyan and Petrosyan obtained the necessary and sufficient conditions for the existence of locally-balanced $2$-partitions of complete multipartite graphs. Recently, Gharibyan \cite{AramYSU} studied locally-balanced $2$-partitions of even and odd graphs. In particular, he gave necessary conditions for the existence of locally-balanced $2$-partitions of these graphs. In \cite{GharibPet1}, Gharibyan and Petrosyan proved that if $G$ is a subcubic bipartite graph that has no cycle of length $2 \pmod{4}$, then $G$ has a locally-balanced $2$-partition with an open neighborhood.

In the present paper we study the complexity of the problem of the existence of locally-balanced $2$-partition with an open (closed) neighborhood of graphs. In particular, we prove that the problem of deciding if a given graph has a locally-balanced $2$-partition with an open neighborhood is $NP$-complete even for $(3,8r)$-biregular bipartite graphs for a fixed $r$ and even bipartite graphs with maximum degree $4$, and the problem of deciding if a given graph has a locally-balanced $2$-partition with a closed neighborhood is $NP$-complete even for subcubic bipartite graphs and odd graphs with maximum degree $3$. We also show that a $(2,2k+1)$-biregular bipartite graph ($k\geq 1$) has a locally-balanced $2$-partition with an open neighborhood if and only if it has no cycle of length $2\pmod{4}$.\\

\section{Notation, definitions and auxiliary results}
\label{Def}
In this section we introduce some terminology and notation. If $G$ is a connected graph, the distance between two vertices $u$ and $v$ in $G$, we denote by $d_G(u,v)$. If $\varphi$ is a $2$-partition of a graph $G$ and $v \in V(G)$, then define $\#(v)_{\varphi}$, $\#[v]_{\varphi}$ and $\varphi^*(v)$ as follows: 
$$\#(v)_{\varphi} = \vert \{u\in N_{G}(v)\colon\,\varphi(u)=0\}\vert - \vert
\{u\in N_{G}(v)\colon\,\varphi(u)=1\}\vert,$$
$$\#[v]_{\varphi} = \vert \{u\in N_{G}[v]\colon\,\varphi(u)=0\}\vert - \vert \{u\in N_{G}[v]\colon\,\varphi(u)=1\}\vert,$$
\begin{center}
   $\varphi^*(v) = \left\{
\begin{tabular}{ll}
$-1$, & if  $\varphi(v) = 0$, \\
$1$, & if $\varphi(v) = 1$. \\
\end{tabular}%
\right.$
\end{center}
 
A spanning subgraph $F$ of a graph $G$ is called an \emph{$[a, b]$-factor} of $G$ if $a \leq d_F(v) \leq b$ for each $v \in V(G)$. We will use the following result from factor theory.

\begin{theorem}\label{mytheorem6}
(\cite{Tutte}) Let $k$ and $r$ be integers such that $1 \leq k < r$. Then every $r$-regular graph (where multiple edges and loops are allowed) has a $[k, k + 1]$-factor.
\end{theorem}

If $F$ is a graph and $\{v_1,\ldots,v_k\}\subseteq V(F)$, then we call a graph $F \left[v_1, \ldots, v_k \right]$ \emph{reduction element} with an input set $\{v_1,\ldots,v_k\}$.

For a reduction element $F \left[v_1, \ldots, v_k \right]$ and a graph $H$ where $V(H) \cap V(F) = \emptyset$, we define a new graph $G=(F \left[v_1, \ldots, v_k \right] \cup H)_{E'}$ as follows:
\begin{gather*}
V(G) = V(F) \cup V(H),\\
E(G) = E(F) \cup E(H) \cup E',
\end{gather*}
where $E' \subseteq	\{uv : v \in \{v_1, \ldots v_k \}, u \in V(H)\}$.

We denote by $V=\{x_1, \ldots, x_n\}$ a finite set of variables. A literal is either a variable $x$ or a negated variable $\overline{x}$. We denote by $L_V=\{x, \overline{x} : x \in V\}$ the set of literals. A \emph{clause} is a set of literals, i.e., a subset of $L_V$, and a \emph{$k$-clause} is one which contains exactly $k$ distinct literals. A clause is \emph{monotone} if all of its involved variables contain no negations.

We define a function $NAE_n : \{0,1\}^n \rightarrow \{0,1\}$ in the following way:
\[
NAE_n(x_1, x_2, \ldots, x_n) = \left\{
\begin{tabular}{ll}
$0$, & if $x_1 = x_2 = \cdots = x_n$,\\
$1$, & otherwise.
\end{tabular}%
\right.
\]
If $c$ is a monotone $k$-clause and $x_{i_1},x_{i_2}, \ldots, x_{i_k} \in c$, then define $NAE_k(c)$ as follows:
\[
NAE_k(c) = NAE_k(x_{i_1},x_{i_2}, \ldots, x_{i_k})
\]

Let us consider the following

\begin{problem}[NAE-3-Sat-E4]\label{myproblem1}

\textit{Instance}: Given a set $V=\{x_1,\ldots,x_n\}$ of variables and a collection $C=\{c_1,\ldots,c_k\}$ of monotone $3$-clauses over $V$ such that every variable appears in exactly four clauses.

\textit{Question}: Is $f(x_1,\ldots,x_n)=NAE_3(c_1) \; \land \; \cdots \; \land \; NAE_3(c_k)$ formula satisfiable?
\end{problem}
In \cite{3Sat}, it was proved the following result.

\begin{theorem}\label{mytheorem1}
Problem \ref{myproblem1} is $NP$-complete.
\end{theorem}

\section{Complexity results on locally-balanced 2-partitions with an open neighborhood}
\label{NPC:Open}

In this section we consider the problem of the existence of locally-balanced $2$-partitions with an open neighborhood of bipartite graphs. In \cite{Kratochvil}, Kratochvil showed that the problem of the existence of the locally-balanced $2$-partition with an open neighborhood of the $(2,2k)$-biregular ($k\geq 2$) bipartite graphs can be solved in polynomial time. Here, we consider locally-balanced $2$-partitions with an open neighborhood of the $(2,2k+1)$-biregular ($k\in \mathbb{N})$) bipartite graphs. We begin our considerations with the following simple lemma.

\begin{lemma}\label{mylemma}
If $\varphi$ is a locally-balanced $2$-partition with an open neighborhood of a graph $G$, then for every $v \in V(G)$ with $d_G(v)=2$, $\varphi(u_1) \neq \varphi(u_2)$, where $vu_1,vu_2 \in E(G) ~(u_1\neq u_2)$.
\end{lemma}

\begin{proof}
Suppose $\varphi$ is a locally-balanced $2$-partition with an open neighborhood of $G$. Let us consider a vertex $v \in V(G),\, where \; d_G(v)=2$. Then
\[
 \#(v)_\varphi = 0 = |\{ u: \; u\in N_G(v), \; \varphi(u)=0\}| - |\{ u: \; u\in N_G(v), \; \varphi(u)=1\}|; 
\]
hence
$|\{ u: \; u\in N_G(v), \; \varphi(u)=0\}| = |\{ u: \; u\in N_G(v), \; \varphi(u)=1\}|$. This implies that $\varphi(u_1) \neq \varphi(u_2)$. \hfill $\Box$
\end{proof}

Our first result was earlier proved in \cite{GharibPet1}, but for the sake of completeness we decide to include it.

\begin{theorem}\label{mytheorem7}
A $(2,2k+1)$-biregular bipartite graph $G$ ($k\in \mathbb{N}$) with bipartition $(X,Y)$ has a locally-balanced $2$-partition with an open neighborhood if and only if it has no cycle of length $2 \pmod{4}$.
\end{theorem}
\begin{proof}
Without loss of generality we may assume that $G$ is connected.
Next, assume that $C=x_{i_1},y_{j_1},x_{i_2},y_{j_2}, \ldots, x_{i_{2r+1}},y_{j_{2r+1}},x_{i_1}$ is a cycle of length of $4r+2$ $(r \geq 1)$, where $x_{i_l} \in X$ and $y_{i_l} \in Y$ $(1 \leq l \leq 2r+1)$. Suppose, to the contrary, that there exists a locally-balanced $2$-partition with an open neighborhood $\varphi$ of $G$. By Lemma \ref{mylemma}, we have
\begin{equation}\label{myequation23}
    \varphi(y_{j_l}) \neq \varphi(y_{j_{(l \bmod (2r+1)) + 1}}) ~(1 \leq l \leq 2r+1).
\end{equation}
Hence
\[
    \varphi(y_{j_1}) = \varphi(y_{j_3}) = \ldots = \varphi(y_{j_{2r+1}}),
\]
which contradicts (\ref{myequation23}) for $l=1$.

Now suppose that $G$ has no cycle of length $4r+2$ $(r \geq 1)$. Let $Y^{(2)}=\{yy': y,y'\in Y, y\neq y'\}$. We define a function $f:X\rightarrow Y^{(2)}$ as follows: for every $x \in X$, let
\[
f(x) = uv \text{, where } u \in N_G(x) \text{ and } \, v \in N_G(x).
\]

Let us now construct a graph $G'$ (where multiple edges are allowed) in the following way:
\begin{gather*}
V(G') = Y,\\
E(G') \text{ contains all possible edges $f(x)$, where $x \in X$}.
\end{gather*}
It is easy to see that $G'$ is a $(2k+1)$-regular graph (where multiple edges are allowed). By Theorem \ref{mytheorem6}, $G'$ has a $[k,k+1]$-factor $H$. Let $\overline{y}\in V(G')$ be a vertex. Now, let us define a $2$-partition $\varphi$ of $G$ by two steps as follows:

\begin{enumerate}
    \item For $x \in X$, let 
    \[ 
        \varphi(x) = \left\{
        \begin{tabular}{ll}
            $1$, & if $f(x) \in E(H)$,\\
            $0$, & otherwise. 
        \end{tabular}%
        \right.
    \]
    \item For $y \in Y$, let  
    \[ 
        \varphi(y) = \left\{
        \begin{tabular}{ll}
            $1$, & if $d_G(y,\overline{y}) \bmod 4 = 0$,\\
            $0$, & if $d_G(y,\overline{y}) \bmod 4 = 2$. 
        \end{tabular}%
        \right.
    \]
\end{enumerate}

Let us show that $\varphi$ is a locally-balanced $2$-partition with an open neighborhood. Let us first consider the vertices of $Y$. Clearly, for any $v \in V(G')$, 
\[
||\{u : uv \in E(H)\}| - |\{u : uv \notin E(H)\}|| \leq 1.
\]
This implies that for any $y \in Y$, 
\[
\left\vert\#(y)_\varphi\right\vert \leq 1.
\]
Let us now consider the vertices of $X$. Suppose, to the contrary, that there exists $x_0 \in X$, where $x_0y_{i_1},x_0y_{i_2}\in E(G)$, $y_{i_1} \neq y_{i_2}$ such that $\varphi(y_{i_1}) = \varphi(y_{i_2})$. This implies that 
\begin{equation}\label{myequation24}
    d_G(\overline{y},y_{i_1}) \bmod 4 = d_G(\overline{y},y_{i_2}) \bmod 4.
\end{equation}

Let $P_1$ be the shortest path between the vertices $y_{i_1}$ and $\overline{y}$ with the length $r_1$ and $P_2$ be the shortest path between the vertices $\overline{y}$ and $y_{i_2}$ with the length $r_2$. From this and taking into account (\ref{myequation24}), we obtain
\begin{equation}\label{myequation25}
    (r_1 + r_2) \bmod 4 = 0.
\end{equation}
Let us consider a closed walk $x_{0},P_1,P_2,x_{0}$. The length of the closed walk is $r_1+r_2+2 = 4p + 2 ~(p \geq 1)$. It can be shown (by induction on the length of the closed walk) that each such a closed walk contains a cycle of length $4t+2$ in a graph $G$ ($t \leq p$), which is a contradiction. \hfill $\Box$
\end{proof}

This result together with the result of Kratochvil \cite{Kratochvil} imply that the problem of the existence of locally-balanced $2$-partition with an open neighborhood of $(a,b)$-biregular bipartite graphs with $\min\{a,b\}\leq 2$ can be solved in polynomial time. On the other hand, in \cite{GharibPet1}, it was shown that the same problem is $NP$-complete for $(3,8)$-biregular bipartite graphs. Here, we prove a more general result. 
Let us consider the following  

\begin{problem}\label{myproblem2}

\textit{Instance}: For a fixed $r$, given a $(3,8r)$-biregular bipartite graph $G$.

\textit{Question}: Does $G$ have a locally-balanced 2-partition with an open neighborhood?
\end{problem}

\begin{theorem}\label{mytheorem2}
Problem \ref{myproblem2} is $NP$-complete.
\end{theorem}
\begin{proof}
It is easy to see that Problem 2 is in $NP$. For the proof of the $NP$-completeness, we show a reduction from Problem 1 to Problem 2. Let $\mathcal{I}=(V,C)$ be an instance of Problem 1, where $V = \{x_1, \ldots, x_n\}$ and $C = \{c_1, \ldots, c_k \}$. We must construct a $(3,8r)$-biregular bipartite graph $G=(X,Y;E)$ such that $G$ has a locally-balanced 2-partition with an open neighborhood if and only if $f(x_1,\ldots,x_n) = NAE_3(c_1) \; \land \; \cdots \; \land \; NAE_3(c_k)$ formula is satisfiable.

Let us construct a graph $G$ in the following way:
\begin{gather*}
X=\{p_1,\ldots,p_n\},\\
Y=\left\{q_1^1, q_1^2,\ldots, q_1^{2r}, q_2^1, q_2^2, \ldots, q_2^{2r}, \ldots, q_k^1, q_k^2,\ldots  q_k^{2r}\right\},\\
E=\left\{p_iq_j^l:\; 1 \leq i \leq n,\, 1 \leq j \leq k,\, 1 \leq l \leq 2r,\, x_i \in c_j\right\}.
\end{gather*}
It is not hard to see that the graph $G$ can be constructed from $V$ and $C$ in polynomial time, we used only $n+2rk$ vertices. We first suppose that $(\beta_1,\ldots,\beta_n)$ is a true assignment of $f(x_1,\ldots,x_n)$. We show that $G$ has a locally-balanced 2-partition with an open neighborhood. 

Let us define a $2$-partition $\varphi$ of $G$ as follows: for every $w \in V(G)$, let
\begin{equation}\label{myequation1} 
\begin{gathered}
\varphi(w) = \left\{
\begin{tabular}{ll}
$\beta_i$, & if $w=p_i$, where $1 \leq i \leq n$,\\
$l$ mod $2$, & if $w=q_i^l$, where $1 \leq i \leq k$, $1 \leq l \leq 2r$.\\
\end{tabular}%
\right.
\end{gathered}
\end{equation}

Let us show that $\varphi$ is a locally-balanced 2-partition with an open neighborhood. Let us consider vertices of the part $X$. Let $p_i \in X$ $(1 \leq i \leq n )$, where $x_i\in c_{j_1}$, $x_i\in c_{j_2}$, $x_i\in c_{j_3}$, $x_i\in c_{j_4}$, $1 \leq j_1,j_2,j_3,j_4\leq k$. From this and taking into account (\ref{myequation1}), we have 
\[
\#(p_i)_\varphi=\sum_{l=1}^{2r} \varphi^*(q^l_{j_1}) + \sum_{l=1}^{2r} \varphi^*(q^l_{j_2}) + \sum_{l=1}^{2r} \varphi^*(q^l_{j_3}) + \sum_{l=1}^{2r} \varphi^*(q^l_{j_4}) = 0.
\]

Let us now consider the vertices of the part $Y$. Let $q^l_i \in Y$ $(1 \leq i \leq k,\, 1\leq l\leq 2r)$, where $c_i = \{x_{j_1},x_{j_2},x_{j_3}\}$. This means $NAE_3(\beta_{j_1},\beta_{j_2},\beta_{j_3}) = 1$. From this, we get that for 
$1 \leq i \leq k,\, 1\leq l\leq 2$,
\[
|\#(q^l_i)_\varphi|=|\varphi^*(p_{j_1}) + \varphi^*(p_{j_2}) + \varphi^*(p_{j_3})| \leq 1.
\]
Conversely, suppose that $\alpha$ is a locally-balanced 2-partition with an open neighborhood of $G$. Let us define an assignment of $f(x_1,\ldots,x_n)$ as follows: $x_i=\alpha(p_i)$ $(1 \leq i \leq n)$. Let $c_i \in C$ $(1 \leq i \leq k)$ and $c_i = \{x_{j_1},x_{j_2},x_{j_3}\}$. From this and taking into account that $|\#(q^1_i)_\alpha| \leq 1$ and $d_G(q^1_i)=3$, we obtain 
\[
NAE_3(x_{j_1},x_{j_2},x_{j_3})=1,
\]
which implies that $f(x_1,\ldots,x_n)$ is satisfiable. \hfill $\Box$
\end{proof} 

For $i\in \mathbb{N}$, let us define a graph $F_1^i$ (see Fig. 1):

\begin{figure}[h]
\begin{center}
\includegraphics[width=0.5\textwidth]{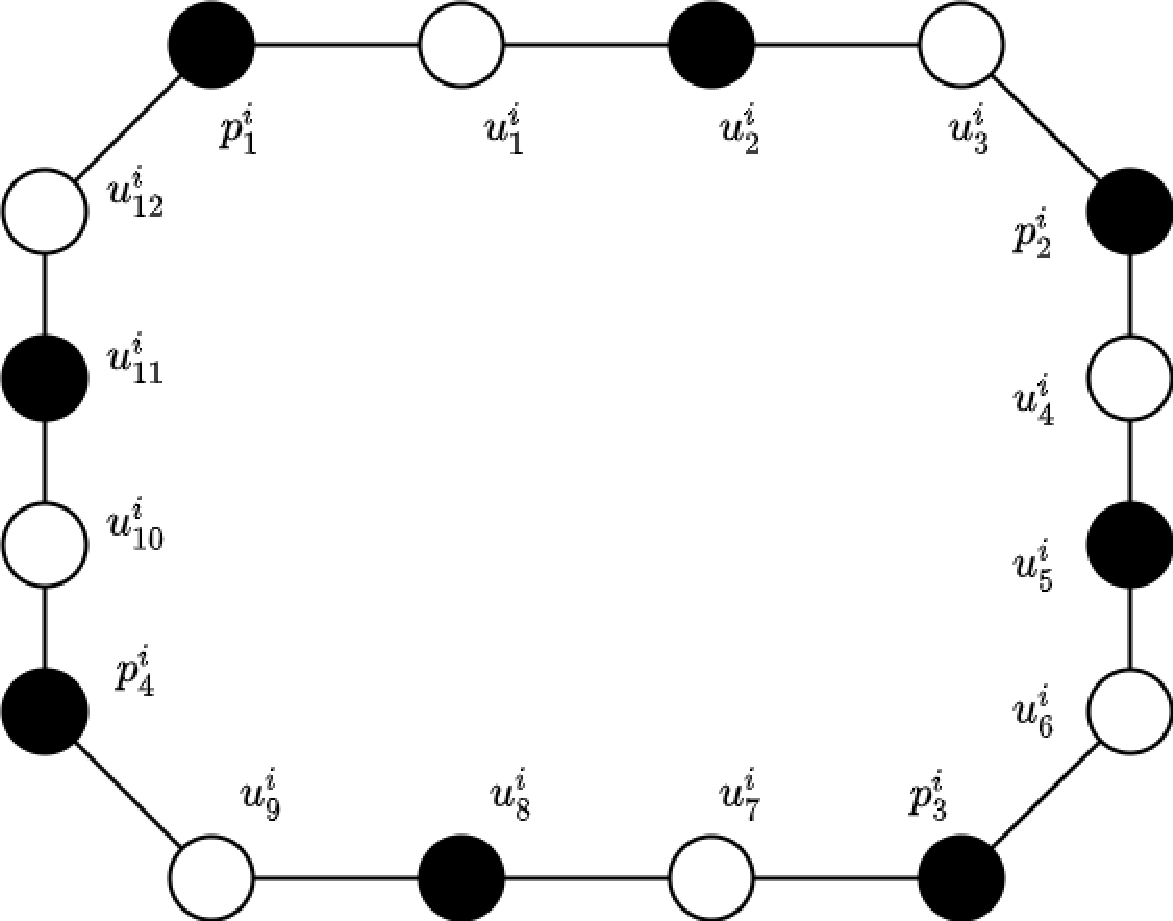}\\
\caption{The graph $F_1^i$.}\label{fig1}
\end{center}
\end{figure}

\begin{lemma}\label{mylemma2}
If a graph $G=(F_1^i \left[p^i_1,p^i_2,p^i_3,p^i_4\right] \cup H)_{E'}$ has $\varphi$ locally-balanced $2$-partition with an open neighborhood, then $\varphi(p^i_1) = \varphi(p^i_2) = \varphi(p^i_3) = \varphi(p^i_4)$.
\end{lemma}
\begin{proof}
Suppose that $\varphi$ is a locally-balanced 2-partition with an open neighborhood of $G$. By Lemma \ref{mylemma}, $\varphi(p^i_1) \neq \varphi(u^i_2)$ and $\varphi(u^i_2) \neq \varphi(p^i_2)$, hence $\varphi(p^i_1) = \varphi(p^i_2)$. Using the same method, we can prove that $\varphi(p^i_1) = \varphi(p^i_2) = \varphi(p^i_3) = \varphi(p^i_4)$. \hfill $\Box$
\end{proof}

As it was discussed before, the problem of the existence of locally-balanced $2$-partition with an open neighborhood of $(2,4)$-biregular bipartite graphs can be solved in polynomial time. Nevertheless, our next complexity result concerns even bipartite graphs and shows that the same problem is $NP$-complete for bipartite graphs where each vertex has degree either $2$ or $4$. Let us now consider the following 
\begin{problem}\label{myproblem3}
 
\textit{Instance}: An even bipartite graph $G$ with $\Delta(G)=4$.

\textit{Question}: Does $G$ have a locally-balanced $2$-partition with an open neighborhood?
\end{problem}

\begin{theorem}\label{mytheorem3}
Problem \ref{myproblem3} is $NP$-complete.
\end{theorem}
\begin{proof}
It is easy to see that Problem 3 is in $NP$. For the proof of the $NP$-completeness, we show a reduction from Problem 1 to Problem 3. Let $\mathcal{I}=(V,C)$ be an instance of Problem 1, where $V = \{x_1, \ldots, x_n\}$ and $C = \{c_1, \ldots, c_k \}$. We must construct an even bipartite graph $G$ such that it has a locally-balanced 2-partition with an open neighborhood if and only if $f(x_1,\ldots,x_n)=NAE_3(c_1) \; \land \; \cdots \; \land \; NAE_3(c_k)$ formula is satisfiable.

Let us construct a graph $G$ in such a way:
\begin{gather*}
V(G) = \left(\bigcup\limits_{i=1}^{n}V(F_1^i \left[p_1^i,p_2^i,p_3^i,p_4^i \right])\right) \cup \left\{q_1^1,q_1^2, q_2^1, q_2^2 \ldots, q_k^1, q_k^2\right\} \cup \{v_1,v_2, \ldots, v_k\},\\
E(G)=\{p_{t}^{i}q_j^l:\; 1 \leq i \leq n,\, 1 \leq j \leq k,\, 1 \leq l \leq 2,\, x_i \in c_j \text{ and}\;\\
\text{it is $x_i$'s $t$-th appearance in the formula $1 \leq t \leq k$}\}\cup \\
\cup \left( \bigcup\limits_{i=1}^{n}E(F_1^i \left[p_1^i,p_2^i,p_3^i,p_4^i\right])\right) \cup \left\{q_i^jv_i:\; 1 \leq i \leq k,\, 1 \leq j \leq 2\right\}.
\end{gather*}

The graph $G$ has $16n+3k$ vertices, so it can be constructed from $V$ and $C$ in polynomial time. Clearly, $G$ is an even bipartite graph. Since each clause contains exactly three distinct literals and $\Delta(F_1^i \left[p_1^i,p_2^i,p_3^i,p_4^i \right])=2$, we have $\Delta(G)=4$. Suppose that $(\beta_1,\ldots,\beta_n)$ is a true assignment of $f(x_1,\ldots,x_n)$. We show that $G$ has a locally-balanced $2$-partition with an open neighborhood. Let us define a $2$-partition $\varphi$ of $G$ by two steps as follows:
\begin{enumerate}
    \item For any $w \in V(G)\; (w \neq v_i)$, 
    \begin{equation}\label{myequation2}
\begin{gathered}
 \quad \varphi(w) = \left\{
\begin{tabular}{lll}
$\beta_i$, & if $w=p^i_l$,& where $1 \leq i \leq n, \, 1 \leq l \leq 4$,\\
$\beta_i$, & if $w=u^i_{3l-2}$, & where $1 \leq i \leq n, \, 1 \leq l \leq 4$,\\
$1-\beta_i$, & if $w=u^i_{3l}$, & where $1 \leq i \leq n, \, 1 \leq l \leq 4$,\\
$1-\beta_i$, & if $w=u^i_{3l-1}$, & where $1 \leq i \leq n, \, 1 \leq l \leq 4$,\\
$1$, & if $w=q_i^1$, & where $1 \leq i \leq k$,\\
$0$, & if $w=q_i^2$, & where $1 \leq i \leq k$.\\
\end{tabular}%
\right.
\end{gathered}
\end{equation}
    \item For any  $v_i \in V(G)$ $(1 \leq i \leq k)$, 
    \[ 
        \varphi^*(v_i) = - \sum_{w\in N_G(q_i) \setminus \{v_i\}}\varphi^*(w).
    \]
\end{enumerate}

Let us show that $\varphi$ is a locally-balanced $2$-partition with an open neighborhood.

\textbf{(a)} If $p^i_l \in V(G)$ $(1 \leq l \leq 4)$, then, let $x_i \in c_j$ and it is $x_i$'s $l$-th appearance in the formula $(1 \leq j \leq k)$. From this and taking into account (\ref{myequation2}), we have
\begin{gather*}
\#(p^i_l)_\varphi= \varphi^*(q^1_{j}) + \varphi^*(q^2_{j}) + \varphi^*(u^i_{3l-2}) + \varphi^*(u^i_{(3(l-1) - 1)\bmod 12 + 1})= 0.
\end{gather*}

\textbf{(b)} If $u^i_{3l-2} \in V(G)$ $(1 \leq l \leq 4)$, then, by (\ref{myequation2}), we have
\[
\#(u^i_{3l-2})_\varphi= \varphi^*(p^i_{l}) + \varphi^*(u^i_{3l-1})= 0.
\]

\textbf{(c)} If $u^i_{3l-1} \in V(G)$ $(1 \leq l \leq 4)$, then, by (\ref{myequation2}), we obtain
\[
\#(u^i_{3l-1})_\varphi= \varphi^*(u^i_{3l-2}) + \varphi^*(u^i_{3l})= 0.
\]

\textbf{(d)} If $u^i_{3l} \in V(G)$ $(1 \leq l \leq 4)$, then, by (\ref{myequation2}), we get
\[
\#(u^i_{3l})_\varphi= \varphi^*(u^i_{3l-1}) + \varphi^*(p^i_{(l\bmod 4) + 1})= 0.
\]

\textbf{(e)} If $q^l_i \in V(G)$ $(1 \leq i \leq k,\, 1\leq l\leq 2)$, then, by definition of $\varphi$, we get
\begin{gather*}
    \#(q^l_i)_\varphi = \sum_{w\in N_G(q_i) \setminus \{v_i\}}\varphi^*(w) + \varphi^*(v_i) = \\
    \sum_{w\in N_G(q_i) \setminus \{v_i\}}\varphi^*(w) - \sum_{w\in N_G(q_i) \setminus \{v_i\}}\varphi^*(w)  = 0.
\end{gather*}

\textbf{(f)} If $v_i \in V(G)$ $(1 \leq i \leq k)$, then, by (\ref{myequation2}), we obtain
\[
\#(v_i)_\varphi =  \varphi^*(q^1_i) + \varphi^*(q^2_i) = 0.
\]

Conversely, suppose that $\alpha$ is a locally-balanced 2-partition with an open neighborhood of $G$. By Lemma \ref{mylemma2}, we have $\alpha(p^i_1)=\alpha(p^i_2)=\alpha(p^i_3)=\alpha(p^i_4)$. Let us define an assignment of $f(x_1,\ldots,x_n)$ as follows: $x_i=\alpha(p^i_1)$ $(1 \leq i \leq n)$. Let $c_i \in C$ $(1 \leq i \leq k)$ and $c_i = \{x_{j_1},x_{j_2},x_{j_3}\}$. From this and taking into account that $|\#(q^1_i)_\alpha| \leq 1$, $\alpha(p^i_1)=\alpha(p^i_2)=\alpha(p^i_3)=\alpha(p^i_4)$ and $d_{G-v_{i}}(q^1_i)=3$, we get 

\[
NAE_3(x_{j_1},x_{j_2},x_{j_3})=1,
\]
which implies that $f(x_1,\ldots,x_n)$ is satisfiable. \hfill $\Box$
\end{proof}

\section{Complexity results on locally-balanced 2-partitions with a closed neighborhood}
\label{NPC:Closed}

In this section we consider the problem of the existence of locally-balanced $2$-partitions with a closed neighborhood of some restricted classes of graphs. 
In \cite{BalKamNP2}, Balikyan and Kamalian proved that the problem of existence of locally-balanced $2$-partition with a closed neighborhood of bipartite graphs with maximum degree $4$ is $NP$-complete, and later they conjectured that the problem remains $NP$-complete even for subcubic bipartite graphs. In this section we confirm the conjecture. We begin our considerations with some auxiliary lemmas.

\begin{lemma}\label{mylemma3}
If $\varphi$ is a locally-balanced $2$-partition with a closed neighborhood of graph $G=(V,E)$, then for every $v \in V(G)$ with $d_G(v)=1$, $\varphi(u) \neq \varphi(v),\;$ where $uv \in E(G)$.
\end{lemma}
\begin{proof}
Suppose that $\varphi$ is a locally-balanced $2$-partition with a closed neighborhood of $G$. Let us consider a vertex $v \in V(G)$ with $d_G(v)=1$. Then
\[
 \#[v]_\varphi = 0 = |\{ u: \; u\in N_G[v] \; and \; \varphi(u) = 0\} - \{ u: \; u\in N_G[v] \; and \; \varphi(u)=1\}|.
\]
Hence,
\[ 
|\{ u: \; u\in N_G[v] \; and \; \varphi(u)=0\} = \{ u: \; u\in N_G[v] \; and \; \varphi(u)=1\}|,
\]

which implies that $\varphi(v) \neq \varphi(u)$. \hfill $\Box$
\end{proof}

Let us define a graph $F_2$ (see Fig. 2):

\begin{figure}[h]
\begin{center}
\includegraphics[width=0.30\textwidth]{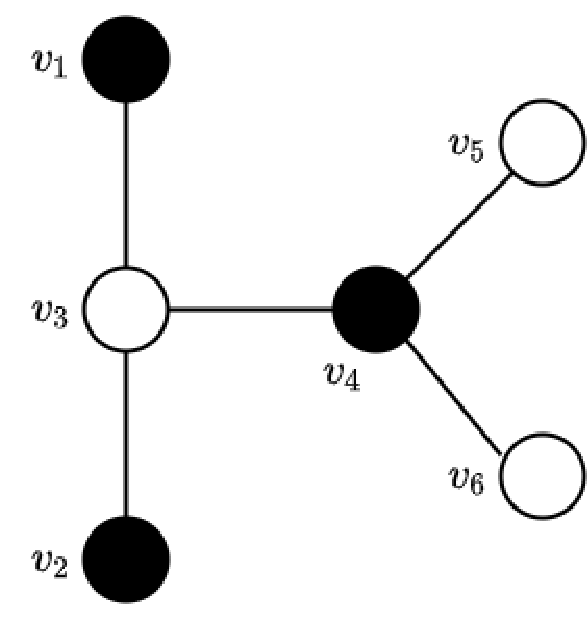}\\
\caption{The graph $F_2$.}\label{fig2}
\end{center}
\end{figure}

\begin{lemma}\label{mylemma4}
\textit{If a graph $F_2$ has $\varphi$ locally-balanced $2$-partition with a closed neighborhood, then $\varphi(v_1) = \varphi(v_2) = \varphi(v_5) = \varphi(v_6)$ and $\varphi(v_3) = \varphi(v_4) = 1 - \varphi(v_1)$. }
\end{lemma}
\begin{proof}
Suppose that $\varphi$ is a locally-balanced $2$-partition with a closed neighborhood of $G$. By Lemma \ref{mylemma3}, we obtain
\begin{gather*}
\varphi(v_5) \neq \varphi(v_4), \\
\varphi(v_6) \neq \varphi(v_4).
\end{gather*}
Hence 
\begin{equation}\label{myequation3}
    \varphi(v_5) = \varphi(v_6).
\end{equation}

Since $d_G(v_4)$ is odd and $\varphi$ is a locally-balanced $2$-partition with a closed neighborhood, we obtain $\#[v_4]_\varphi=0$. From this and taking into account (\ref{myequation3}), we have
\begin{equation}\label{myequation4}
    \varphi(v_4) = \varphi(v_3) \neq \varphi(v_5).
\end{equation}

Using the same assumption for $v_3$ and taking into account (\ref{myequation4}), we get
\[
    \varphi(v_1) = \varphi(v_2) \neq \varphi(v_3).
\] \hfill $\Box$
\end{proof}

For $i\in \mathbf{N}$, let us define a graph $F_3^i$ (see Fig. 3):

\begin{figure}[h]
\begin{center}
\includegraphics[width=0.90\textwidth]{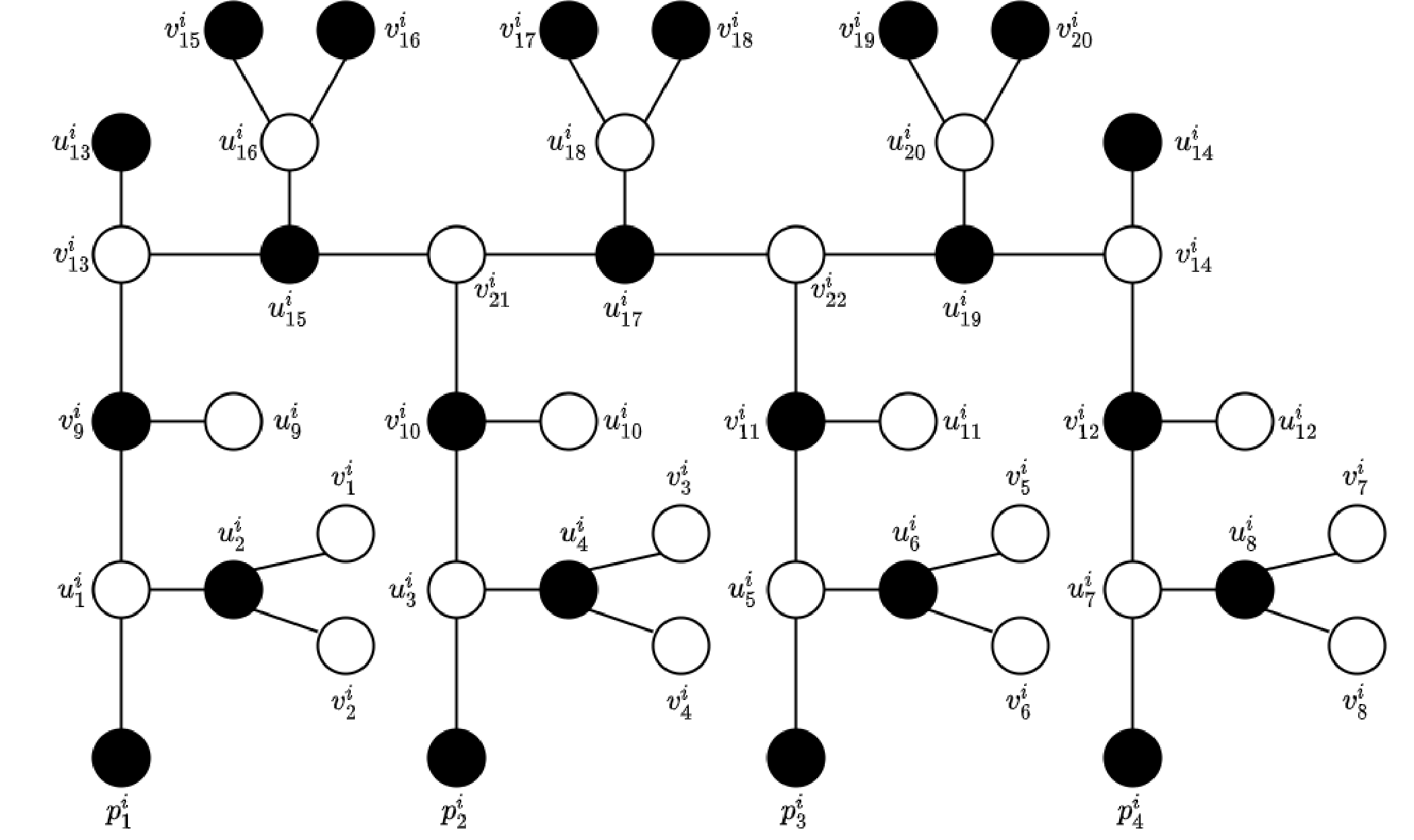}\\
\caption{The graph $F_3^i$.}\label{fig3}
\end{center}
\end{figure}

\begin{lemma}\label{mylemma5}
If a graph $G=(F_3^i \left[p^i_1,p^i_2,p^i_3,p^i_4 \right] \cup H)_{E'}$ has $\varphi$ locally-balanced 2-partition with a closed neighborhood, then $\varphi(p^i_1) = \varphi(p^i_2) = \varphi(p^i_3) = \varphi(p^i_4)$.
\end{lemma}
\begin{proof}
Suppose that $\varphi$ is a locally-balanced 2-partition with a closed neighborhood of $G$. Without loss of generality we may assume that $\varphi(u^i_1)=1$. By Lemma \ref{mylemma4}, we obtain

\begin{equation}\label{myequation5}
\begin{gathered}
\varphi(u^i_2) = \varphi(u^i_1) = 1, \\
\varphi(v^i_1) = \varphi(v^i_2) = \varphi(p^i_1) = \varphi(v^i_9) = 0.
\end{gathered}
\end{equation}
By Lemma \ref{mylemma3} and taking into account (\ref{myequation5}), we have
\begin{equation}\label{myequation6}
    \varphi(u^i_9) = 1 - \varphi(v^i_9) = 1.
\end{equation}
From (\ref{myequation5}) and (\ref{myequation6}), we get
\begin{equation}\label{myequation7}
    \varphi(v^i_{13}) = \varphi(v^i_9) = 0.
\end{equation}
By Lemma \ref{mylemma3} and taking into account (\ref{myequation7}), we get
\begin{equation}\label{myequation8}
    \varphi(u^i_{13}) = 1 - \varphi(v^i_{13}) = 1. 
\end{equation}
By Lemma \ref{mylemma4} and (\ref{myequation8}), we have
\begin{equation}\label{myequation9}
\begin{gathered}
\varphi(u^i_{15}) = \varphi(u^i_{16}) = 1,\\
\varphi(v^i_{15}) = \varphi(v^i_{16}) = \varphi(v^i_{21}) = 0.
\end{gathered}
\end{equation}
By Lemma \ref{mylemma4} and taking into account (\ref{myequation9}), we obtain
\begin{equation}\label{myequation10}
\begin{gathered}
\varphi(u^i_{17}) = \varphi(u^i_{18}) = 1.\\
\varphi(v^i_{17}) = \varphi(v^i_{18}) = \varphi(v^i_{22}) = 0.
\end{gathered}
\end{equation}
By Lemma \ref{mylemma4} and (\ref{myequation10}), we have
\begin{equation}\label{myequation11}
\begin{gathered}
\varphi(u^i_{19}) = \varphi(u^i_{20}) = 1,\\
\varphi(v^i_{14}) = \varphi(v^i_{19}) = \varphi(v^i_{20}) = 0.
\end{gathered}
\end{equation}
From (\ref{myequation9}) and (\ref{myequation10}), we get
\begin{equation}\label{myequation12}
    \varphi(v^i_{21}) = \varphi(v^i_{10}) = 0.
\end{equation}
By Lemma \ref{mylemma3} and taking into account (\ref{myequation12}), we obtain
\begin{equation}\label{myequation13}
    \varphi(u^i_{10}) = 1 - \varphi(v^i_{10}) = 1.
\end{equation}
By Lemma \ref{mylemma4} and taking into account (\ref{myequation13}), we obtain
\begin{equation}\label{myequation14}
\begin{gathered}
\varphi(u^i_3) = \varphi(u^i_4) = 1,\\
\varphi(v^i_3) = \varphi(v^i_4) = \varphi(p^i_2) = 0.
\end{gathered}
\end{equation}
From (\ref{myequation10}) and (\ref{myequation11}), we get
\begin{equation}\label{myequation15}
    \varphi(v^i_{11}) = \varphi(v^i_{22}) = 0.
\end{equation}
By Lemma \ref{mylemma3} and taking into account (\ref{myequation15}), we have
\begin{equation}\label{myequation16}
    \varphi(u^i_{11}) = 1 - \varphi(v^i_{11}) = 1.
\end{equation}
By Lemma \ref{mylemma4} and (\ref{myequation16}), we obtain
\begin{equation}\label{myequation17}
\begin{gathered}
\varphi(u^i_5) = \varphi(u^i_6) = 1,\\
\varphi(v^i_5) = \varphi(v^i_6) = \varphi(p^i_3) = 0.
\end{gathered}
\end{equation}
By Lemma \ref{mylemma3} and (\ref{myequation11}), we have
\begin{equation}\label{myequation18}
    \varphi(u^i_{14}) = 1 - \varphi(v^i_{14}) = 1.
\end{equation}
From (\ref{myequation11}) and (\ref{myequation18}), we get
\begin{equation}\label{myequation19}
    \varphi(v^i_{12}) = \varphi(v^i_{14}) = 0.
\end{equation}
By Lemma \ref{mylemma3} and (\ref{myequation19}), we have
\begin{equation}\label{myequation20}
    \varphi(u^i_{12}) = 1 - \varphi(v^i_{12}) = 1.
\end{equation}
By Lemma \ref{mylemma4} and taking into account (\ref{myequation20}), we obtain
\begin{equation}\label{myequation21}
\begin{gathered}
\varphi(u^i_7) = \varphi(u^i_8) = 1,\\
\varphi(v^i_7) = \varphi(v^i_8) = \varphi(p^i_4) = 0.
\end{gathered}
\end{equation}
By (\ref{myequation5}), (\ref{myequation14}), (\ref{myequation17}) and (\ref{myequation21}), we have
\[
    \varphi(p^i_1) = \varphi(p^i_2) = \varphi(p^i_3) = \varphi(p^i_4) = 0.
\] \hfill $\Box$
\end{proof}

From the proof of Lemma \ref{mylemma5}, it follows the following result.
\begin{lemma}\label{mylemma6}
If $\varphi$ is a $2$-partition of a graph $G=(F_3^i \left[p^i_1,p^i_2,p^i_3,p^i_4 \right] \cup H)_{E'}$, $\varphi(p^i_j)=\varphi(v^i_l)$ $(1 \leq j \leq 4, \; 1 \leq l \leq 22)$ and $\varphi(u^i_j) = 1 - \varphi(p^i_1)$ $(1 \leq j \leq 20)$, then $\#[v^i_j]_\varphi = 0$ $(1 \leq j \leq 22)$, $\#[u^i_j]_\varphi = 0$ $(1 \leq j \leq 20)$ and $\#[p^i_j]_\varphi = \sum_{w \in \{a\;:\; ap^i_j\in E'\}}\varphi^*(w)$ $(1 \leq j \leq 4)$.
\end{lemma}

Let us now consider the following

\begin{problem}\label{myproblem4}

\textit{Instance}: A bipartite graph $G$ with $\Delta(G)=3$.

\textit{Question}: Does $G$ have a locally-balanced $2$-partition with a closed neighborhood?
\end{problem}

\begin{theorem}\label{mytheorem4}
Problem \ref{myproblem4} is $NP$-complete.
\end{theorem}
\begin{proof}
It is easy to see that Problem 4 is in $NP$. For the proof of the $NP$-completeness, we show a reduction from Problem 1 to Problem 4. Let $\mathcal{I}=(V,C)$ be an instance of Problem 1, where $V = \{x_1, \ldots, x_n\}$ and $C = \{c_1, \ldots, c_k \}$. We must construct a bipartite graph $G$ such that $G$ has a locally-balanced 2-partition with a closed neighborhood if and only if $f(x_1,\ldots,x_n)=NAE_3(c_1) \; \land \; \cdots \; \land \; NAE_3(c_k)$ formula is satisfiable.
Let us construct a graph $G$ in such a way:
\begin{gather*}
V(G) = \left(\bigcup\limits_{i=1}^{n}V(F_3^i \left[p_1^i,p_2^i,p_3^i,p_4^i \right])\right) \cup \{q_1, q_2 \ldots, q_k\},\\
E(G)=\left\{p_{t}^{i}q_j:\; 1 \leq i \leq n, 1 \leq j \leq k, x_i \in c_j \; \text{and it is $x_i$'s $t$-th appearance in the formula}\right\}.
\end{gather*}

It is not hard to see that the graph $G$ can be constructed from $V$ and $C$ in polynomial time, since we used only $46n+k$ vertices. Clearly, $G$ is a bipartite graph. Since each clause contains exactly three distinct literals and $\Delta(F_3^i \left[p_1^i,p_2^i,p_3^i,p_4^i \right])=3$, we have $\Delta(G)=3$. Suppose $(\beta_1,\ldots,\beta_n)$ is a true assignment of $f(x_1,\ldots,x_n)$. We show that $G$ has a locally-balanced $2$-partition with a closed neighborhood. Let us define a $2$-partition $\varphi$ of $G$ by two steps as follows:

\begin{enumerate}
    \item For any $w \in V(F_3^i \left[p_1^i,p_2^i,p_3^i,p_4^i \right])$ $(1 \leq i \leq n)$, 
    \[ 
        \varphi(w) = \left\{
        \begin{tabular}{lll}
            $\beta_i$, & if $w=p^i_l$,& where $1 \leq l \leq 4$,\\
            $\beta_i$, & if $w=v^i_l$, & where $1 \leq l \leq 22$,\\
            $1-\beta_i$, & if $w=u^i_l$, & where $1 \leq l \leq 20$.\\
        \end{tabular}%
        \right.
    \]
    \item For any  $q_i \in V(G)$ $(1 \leq i \leq k)$, 
    \[ 
        \varphi^*(q_i) = - \sum_{w\in N_G(q_i)}\varphi^*(w).
    \]
\end{enumerate}

Let us show, that $\varphi$ is a locally-balanced $2$-partition with a closed neighborhood.

\textbf{(a)} If $v^i_l \in V(G)$ $(1 \leq l \leq 22)$, then, by the definition of $\varphi$ and Lemma \ref{mylemma6}, we obtain
\[
\#[v^i_l]_\varphi = 0.
\]

\textbf{(b)} If $u^i_l \in V(G)$ $(1 \leq l \leq 20)$, then, by definition of $\varphi$ and Lemma \ref{mylemma6}, we have
\[
\#[u^i_l]_\varphi = 0.
\]

\textbf{(c)} If $p^i_l \in V(G)$ $(1 \leq l \leq 4)$, then, by definition of $\varphi$, Lemma \ref{mylemma6} and taking into account that vertex $p^i_l$ has only one neighbor outside of the subgraph $F_3^i \left[p_1^i,p_2^i,p_3^i,p_4^i \right]$, we get 
\[
|\#[p^i_l]_\varphi| \leq 1.
\]

\textbf{(d)} If $q_i \in V(G)$ $(1 \leq i \leq k)$, then, by definition of $\varphi$, we get
\[
   \varphi^*(q_i) = - \sum_{w\in N_G(q_i)}\varphi^*(w),
\]
which implies that
\[
\#[q_i]_\varphi = - \sum_{w\in N_G(q_i)}\varphi^*(w) + \sum_{w\in N_G(q_i)}\varphi^*(w) = 0.
\]

Conversely, suppose that $\alpha$ is a locally-balanced $2$-partition with a closed neighborhood of $G$. By Lemma \ref{mylemma5}, we have $\alpha(p^i_1)=\alpha(p^i_2)=\alpha(p^i_3)=\alpha(p^i_4)$. Let us define an assignment of $f(x_1,\ldots,x_n)$ as follows: $x_i=\alpha(p^i_1)$ $(1 \leq i \leq n)$. Let $c_i \in C$ $(1 \leq i \leq k)$ and $c_i = \{x_{j_1},x_{j_2},x_{j_3}\}$. From this and taking into account that $|\#(q_i)_\alpha| \leq 1$, $\alpha(p^i_1)=\alpha(p^i_2)=\alpha(p^i_3)=\alpha(p^i_4)$ and $d_{G}(q_i)=3$, we obtain 

\[
NAE_3(x_{j_1},x_{j_2},x_{j_3})=1,
\]
which implies that $f(x_1,\ldots,x_n)$ is satisfiable. \hfill $\Box$
\end{proof}

For $i_1,l_1,i_2,l_2,i_3,l_3 \in \mathbf{N}$, let us define a graph $F_4^{i_1,l_1,i_2,l_2,i_3,l_3}$ (see Fig. 4):

\begin{figure}[h]
\begin{center}
\includegraphics[width=0.40\textwidth]{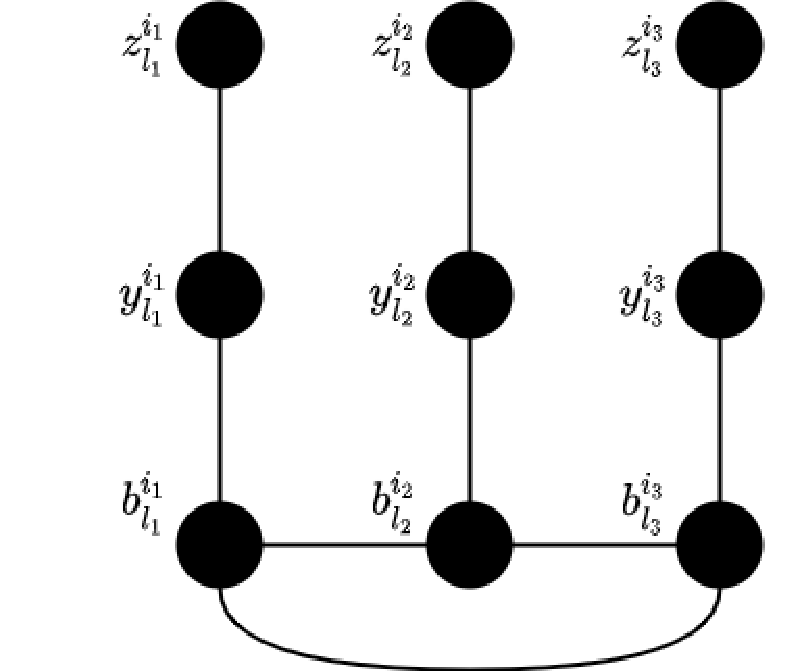}\\
\caption{The graph $F_4^{i_1,l_1,i_2,l_2,i_3,l_3}$.}\label{fig4}
\end{center}
\end{figure}

Finally, we formulate and prove the $NP$-completeness of another problem on locally balanced $2$-partitions with a closed neighborhood.

\begin{problem}\label{myproblem5}

Instance: An odd graph $G$ with $\Delta(G)=3$.

Question: Does $G$ have a locally-balanced $2$-partition with a closed neighborhood?\\
\end{problem}

\begin{theorem}\label{mytheorem5}
Problem \ref{myproblem5} is $NP$-complete.
\end{theorem}
\begin{proof}
It is easy to see that Problem 5 is in $NP$. For the proof of the $NP$-completeness, we show a reduction from Problem 1 to Problem 5. Let $\mathcal{I}=(V,C)$ be an instance of Problem 1, where $V = \{x_1, \ldots, x_n\}$ and $C = \{c_1, \ldots, c_k \}$. We must construct an odd graph $G$ such that $G$ has a locally-balanced $2$-partition with a closed neighborhood if and only if $f(x_1,\ldots,x_n)=NAE_3(c_1) \; \land \; \cdots \; \land \; NAE_3(c_k)$ formula is satisfiable.
Let us construct a graph $G$ in such a way:
\begin{gather*}
V(G) = \left(\bigcup\limits_{i=1}^{n}V(F_3^i \left[p_1^i,p_2^i,p_3^i,p_4^i \right])\right) \cup \{q_1, q_2 \ldots, q_k\}\cup\\
\cup \{w:\;w \in V(F_4^{i_1,l_1,i_2,l_2,i_3,l_3}), \text{where} \; c_j = \{x_{i_1},x_{i_2},x_{i_3}\} \;\\
\text{and it is $x_{i_1}$'s $l_1$-th, $x_{i_2}$'s $l_2$-th and $x_{i_3}$'s $l_3$-th}\\
\text{appearances in the formula, } 1 \leq j \leq k\},\\
E(G)=\{e, q_j p^{i_1}_{l_1},q_j p^{i_2}_{l_2},q_j p^{i_3}_{l_3},y^{i_1}_{l_1} p^{i_1}_{l_1},y^{i_2}_{l_2} p^{i_2}_{l_2},y^{i_3}_{l_3} p^{i_3}_{l_3} : \;e \in E(F_4^{i_1,l_1,i_2,l_2,i_3,l_3}),\\
\text{where} \; c_j = \{x_{i_1},x_{i_2},x_{i_3}\} \; \text{and it is $x_{i_1}$'s $l_1$-th, $x_{i_2}$'s $l_2$-th and $x_{i_3}$'s $l_3$-th}\\
\text{appearances in the formula, } 1 \leq j \leq k \} \cup \left(\bigcup\limits_{i=1}^{n}E(F_3^i \left[p_1^i,p_2^i,p_3^i,p_4^i\right])\right).
\end{gather*}

The graph $G$ has $46n+10k$ vertices, so it can be constructed from $V$ and $C$ in polynomial time. Clearly, $G$ is an odd graph. Since each clause contains exactly three distinct literals, $\Delta(F_3^i \left[p_1^i,p_2^i,p_3^i,p_4^i \right])=3$, $\Delta(F_4^{i_1,l_1,i_2,l_2,i_3,l_3})=3$ and by the construction of $G$, we have $\Delta(G)=3$. Suppose $(\beta_1,\ldots,\beta_n)$ is a true assignment of $f(x_1,\ldots,x_n)$. We show that $G$ has a locally-balanced $2$-partition with a closed neighborhood. Let us define a $2$-partition $\varphi$ of $G$ by three steps as follows:

\begin{enumerate}
    \item For any $w \in V(F_3^i \left[p_1^i,p_2^i,p_3^i,p_4^i \right])$ $(1 \leq i \leq n)$ 
    \[ 
        \varphi(w) = \left\{
        \begin{tabular}{lll}
            $\beta_i$, & if $w=p^i_l$,& where $1 \leq l \leq 4$,\\
            $\beta_i$, & if $w=v^i_l$, & where $1 \leq l \leq 22$,\\
            $1-\beta_i$, & if $w=u^i_l$, & where $1 \leq l \leq 20$.\\
        \end{tabular}%
        \right.
    \]
    \item For any $q_i \in V(G)$ $(1 \leq i \leq k)$ 
    \[ 
        \varphi^*(q_i) = - \sum_{w\in N_G(q_i)}\varphi^*(w).
    \]
    \item For any $w \in V(F_4^{i_1,l_1,i_2,l_2,i_3,l_3})$ $(1 \leq i_1,i_2,i_3 \leq n, \, 1 \leq l_1,l_2,l_3 \leq 4)$
    \[ 
        \varphi(w) = \left\{
        \begin{tabular}{lll}
            $1 - \varphi(q_j)$, & if $w=y^{i_t}_{l_t}$ and $p^{i_t}_{l_t} q_j \in E(G)$, & where $1 \leq t \leq 3$, $1 \leq j \leq k$,\\
            $\varphi(q_j)$, & if $w=z^{i_t}_{l_t}$ and $p^{i_t}_{l_t} q_j \in E(G)$, & where $1 \leq t \leq 3$, $1 \leq j \leq k$,\\
            $1 - \varphi(p^{i_t}_{l_t})$, & if $w=b^{i_t}_{l_t}$ and $p^{i_t}_{l_t} q_j \in E(G)$, & where $1 \leq t \leq 3$, $1 \leq j \leq k$.\\
        \end{tabular}%
        \right.
    \]
\end{enumerate}

Let us show that $\varphi$ is a locally-balanced $2$-partition with a closed neighborhood. 

\textbf{(a)} If $v^i_l \in V(G)$ $(1 \leq l \leq 22)$, then, by definition of $\varphi$ and Lemma \ref{mylemma6}, we have
\[
\#[v^i_l]_\varphi = 0.
\]

\textbf{(b)} If $u^i_l \in V(G)$ $(1 \leq l \leq 20)$, then, by definition of $\varphi$ and Lemma \ref{mylemma6}, we get
\[
\#[u^i_l]_\varphi = 0.
\]

\textbf{(c)} If $p^i_l \in V(G)$ and $p^i_lq_j\in E(G)$ $(1 \leq l \leq 4)$, then, by definition of $\varphi$ and Lemma \ref{mylemma6}, we obtain 
\begin{align*}
\#[p^i_l]_\varphi = \varphi^*(q_j) + \varphi^*(y^i_l) = 0.
\end{align*}

\textbf{(d)} If $q_i \in V(G)$ $(1 \leq i \leq k)$, then, By definition of $\varphi$, we get
\[
   \varphi^*(q_i) = - \sum_{w\in N_G(q_i)}\varphi^*(w),
\]
which implies that
\begin{equation}\label{myequation22}
\#[q_i]_\varphi = - \sum_{w\in N_G(q_i)}\varphi^*(w) + \sum_{w\in N_G(q_i)}\varphi^*(w) = 0.
\end{equation}

\textbf{(e)} If $z^{i_t}_{l_t} \in V(F_4^{i_1,l_1,i_2,l_2,i_3,l_3})$ and $p^{i_t}_{l_t} q_j \in E(G)$ $(1 \leq t \leq 3)$, then, by definition of $\varphi$, we get
\[
\#[z^{i_t}_{l_t}]_\varphi = \varphi^*(z^{i_t}_{l_t}) + \varphi^*(y^{i_t}_{l_t}) = \varphi^*(q_j) + (-\varphi^*(q_j)) = 0.
\]

\textbf{(f)} If $y^{i_t}_{l_t} \in V(F_4^{i_1,l_1,i_2,l_2,i_3,l_3})$ and $p^{i_t}_{l_t} q_j \in E(G)$ $(1 \leq t \leq 3)$, then, by definition of $\varphi$, we obtain
\begin{align*}
\#[y^{i_t}_{l_t}]_\varphi = \varphi^*(y^{i_t}_{l_t}) + \varphi^*(z^{i_t}_{l_t}) + \varphi^*(p^{i_t}_{l_t}) + \varphi^*(b^{i_t}_{l_t}) \\
= -\varphi^*(q_j) + \varphi^*(q_j) + \varphi^*(p^{i_t}_{l_t}) + (-\varphi^*(p^{i_t}_{l_t})) = 0.
\end{align*}

\textbf{(g)} If $b^{i_t}_{l_t} \in V(F_4^{i_1,l_1,i_2,l_2,i_3,l_3})$ and $p^{i_t}_{l_t} q_j \in E(G)$ $(1 \leq t \leq 3)$, then, by definition of $\varphi$ and (\ref{myequation22}), we have 
\begin{align*}
\#[b^{i_t}_{l_t}]_\varphi = \varphi^*(b^{i_1}_{l_1}) + \varphi^*(b^{i_2}_{l_2}) + \varphi^*(b^{i_3}_{l_3}) + \varphi^*(y^{i_t}_{l_t}) \\
= -\varphi^*(p^{i_1}_{l_1}) - \varphi^*(p^{i_2}_{l_2}) - \varphi^*(p^{i_3}_{l_3}) - \varphi^*(q_j) = -\#[q_j]_\varphi = 0.
\end{align*}

Conversely, suppose that $\alpha$ is a locally-balanced $2$-partition with an open neighborhood of $G$. By Lemma \ref{mylemma5}, we have $\alpha(p^i_1)=\alpha(p^i_2)=\alpha(p^i_3)=\alpha(p^i_4)$. Let us define an assignment of $f(x_1,\ldots,x_n)$ as follows: $x_i=\alpha(p^i_1)$ $(1 \leq i \leq n)$. Let $c_i \in C$ $(1 \leq i \leq k)$ and $c_i = \{x_{j_1},x_{j_2},x_{j_3}\}$. From this and taking into account that $|\#(q_i)_\alpha| \leq 1$, $\alpha(p^i_1)=\alpha(p^i_2)=\alpha(p^i_3)=\alpha(p^i_4)$ and $d_G(q_i)=3$, we obtain 

\[
NAE_3(x_{j_1},x_{j_2},x_{j_3})=1,
\]
which implies that $f(x_1,\ldots,x_n)$ is satisfiable. \hfill $\Box$
\end{proof}
\bigskip

\section{Concluding remarks}

In this paper we studied the problem of the existence of locally-balanced $2$-partition with an open (closed) neighborhood. In particular, we proved:

\begin{enumerate}
\item the problem of deciding if a given graph has a locally-balanced $2$-partition with an open neighborhood is $NP$-complete for biregular bipartite graphs and even bipartite graphs with maximum degree $4$;

\item the problem of deciding if a given graph has a locally-balanced $2$-partition with a closed neighborhood is $NP$-complete even for subcubic bipartite graphs and odd graphs with maximum degree $3$.
\end{enumerate}

Although the problem of the existence of locally-balanced $2$-partition with an open (closed) neighborhood is $NP$-complete even for subcubic bipartite graphs, the complexity status of the problem is unknown, for example, for subcubic bipartite planar graphs or planar odd (even) graphs. We think that these classes of graphs are good subjects for further considerations. Another direction for further considerations are locally-balanced $k$-partitions ($k\geq 2$) with an open (closed) neighborhood, which were introduced in \cite{GharibPet2}. In particular, the complexity status of the problem of the existence of locally-balanced $3$-partitions with an open (closed) neighborhood of graphs is unknown.
\bigskip

\bibliographystyle{elsarticle-num}
\bibliography{<your-bib-database>}

\begin{thebibliography}{00}

\bibitem{AndreevRacke} K. Andreev, H. R\"acke, Balanced Graph Partitioning, Proceedings of the Sixteenth Annual ACM Symposium on Parallelism in Algorithms and Architectures, Barcelona, Spain, (2004), pp. 120--124 

\bibitem{BalKamNP1}	
S.V. Balikyan, R.R. Kamalian, On $NP$-Completeness of the problem of existence of locally-balanced $2$-partition for bipartite graphs $G$ with $\Delta(G)=3$, Doklady NAN RA 105 : 1 (2005) 21--27 

\bibitem{BalKamNP2}
S.V. Balikyan, R.R. Kamalian, On $NP$-completeness of the problem of existence of locally-balanced $2$-partition for bipartite graphs $G$ with $\Delta(G)=4$ under the extended definition of the neighborhood of a vertex, Doklady NAN RA, 106 : 3 (2006) 218--226

\bibitem{Bal}
S.V. Balikyan, On existence of certain locally-balanced $2$-partition of a tree, Mathematical Problems of Computer Science 30 (2008) 25--30 

\bibitem{BalKam}
S.V. Balikyan, R.R. Kamalian, On existence of $2$-partition of a tree, which obeys the given priority, Mathematical Problems of Computer Science 30 (2008) 31--35

\bibitem{BazTuzaVand} C. Bazgan, Zs. Tuza, D. Vanderpooten, The satisfactory partition problem, Discrete Applied Mathematics 154 (2006) 1236--1245 

\bibitem{Berge}
C. Berge, Graphs and Hypergraphs. Elsevier Science Ltd, 1985

\bibitem{ChartZhang}
G. Chartrand, P. Zhang, Chromatic Graph Theory, Discrete Mathematics and Its Applications, CRC Press, 2009

\bibitem{3Sat}
A. Darmann, J. Döcker, On a simple hard variant of Not-All-Equal 3-Sat, Theoretical Computer Science 815 (2020) 147--152 

\bibitem{GerKob1} M. Gerber, D. Kobler, Partitioning a graph to satisfy all vertices, Technical report, Swiss Federal Institute of Technology, Lausanne, 1998

\bibitem{GerKob2} M. Gerber, D. Kobler, Algorithmic approach to the satisfactory graph partitioning problem, European J. Oper. Res. 125 (2000) 283--291 

\bibitem{GharibPet}
A.H. Gharibyan, P.A. Petrosyan, Locally-balanced $2$-partitions of complete multipartite graphs, Mathematical Problems of Computer Science 49 (2018) 7--17

\bibitem{AramYSU}
A.H. Gharibyan, On Locally-balanced $2$-partitions of some classes of graphs, Proceedings of the Yerevan State University, Physical and Mathematical Sciences 54 : 1 (2020) 9--19 

\bibitem{GharibPet1} A.H. Gharibyan, P.A. Petrosyan, On locally-balanced $2$-partitions of bipartite graphs, Proceedings of the YSU, Physical and Mathematical Sciences 54: 3 (2020) 137--145

\bibitem{GharibPet2} A.H. Gharibyan, P.A. Petrosyan, Locally-balanced $k$-partitions of graphs, Proceedings of the Yerevan State University, Physical and Mathematical Sciences 55: 2 (2021) 96--112

\bibitem{Ghouila-Houri}
A. Ghouila-Houri, Caracterisation des matrices totalement unimodulaires,
C. R. Acad. Sci. Paris 254 (1962) 1192--1194 

\bibitem{HajSzem}
A. Hajnal, E. Szemeredi, Proof of a Conjecture of P. Erd\H{o}s, Combinatorial Theory and Its Applications, II Proc. Colloq., Balatonfüred (1969). North-Holland, (1970), pp. 601--623 

\bibitem{Kostochka}
A.V. Kostochka, Equitable colorings of outerplanar graphs, Discrete Mathematics 258 (2002) 373--377 

\bibitem{Kratochvil}
J. Kratochvil, Complexity of Hypergraph Coloring and Seidel's Switching, Graph Theoretic Concepts in Computer Science, 29th International Workshop, WG 2003, Elspeet, The Netherlands, Revised Papers, 2880, (2003), pp. 297--308 

\bibitem{Meyer}
W. Meyer, Equitable Coloring, American Mathematical Monthly 80 : 8 (1973) 920--922 

\bibitem{Tutte}
W.T. Tutte, The subgraph probelm, Annals of Discrete Mathematics 3 (1978) 289--295 

\bibitem{deWerra}	
D. de Werra, On Good and Equitable Colorings, In Cahiers du C.E.R.O. 17 (1975) 417--426 

\bibitem{West}
D.B. West, Introduction to Graph Theory, Prentice-Hall, New Jersey, 2001
\end{thebibliography}



\end{document}